\documentclass[12pt,a4paper]{amsart}
\usepackage{fullpage}
\usepackage{url}

\usepackage{algpseudocode,algorithm,algorithmicx}
\newcommand*\Let[2]{\State #1 $\gets$ #2}

\newtheorem{theorem}{Theorem}[section]
\newtheorem{corollary}[theorem]{Corollary}

\newtheorem{lemma}[theorem]{Lemma}

\newtheorem{proposition}[theorem]{Proposition}

\theoremstyle{definition}

\title{Enumeration of Seidel matrices}
\author{Ferenc Sz\"oll\H{o}si and Patric R.J. \"Osterg\aa rd}
\date{\today. Preprint. This research was supported in part by the Academy of Finland, Grant \#289002}
\address{F. Sz. and P.R.J. \"O.: Department of Communications and Networking, Aalto University School of Electrical Engineering, P.O. Box 15400, 00076 Aalto, Finland}
\email{szoferi@gmail.com, patric.ostergard@aalto.fi}

\begin{document}
\begin{abstract}
In this paper Seidel matrices are studied, and their spectrum and several related algebraic properties are determined for order $n\leq 13$. Based on this Seidel matrices with exactly three distinct eigenvalues of order $n\leq 23$ are classified. One consequence of the computational results is that the maximum number of equiangular lines in $\mathbb{R}^{12}$ with common angle $1/5$ is exactly $20$.
\end{abstract}
\maketitle

\section{Introduction}
A Seidel matrix $S$ of order $n$ is an $n\times n$ symmetric $(-1,0,1)$-matrix with $0$ diagonal and $\pm1$ entries otherwise. Seidel matrices were introduced in \cite{vLS} in connection with equiangular lines in Euclidean spaces \cite{HAA}, and were subject of active research during the late $1970$s \cite{C}, \cite{LS}, \cite{T}, \cite{TTH}. Recently, there has been a renewed interest in equiangular lines, see \cite{AM}, \cite{BMV}, \cite{BUKH2}, \cite{ABARG}, \cite{BU}, \cite{DC}, \cite{FS}, \cite{RND}, \cite{GR}, \cite{GRWSN}, \cite{GRWSN2}, \cite{HAM}, \cite{MAK}, \cite{OKUDA} because new results emerged both from new theoretical insights and from computer calculations. In this paper we study small Seidel matrices and their algebraic properties via computer-aided methods. Similar studies were done earlier in \cite{BMS}, \cite{GRWSN2}, and \cite{vLS}.

Two Seidel matrices, $S_1$ and $S_2$, are equivalent, denoted by $S_1\sim S_2$, if there exists a signed permutation matrix $P$ such that $S_2=PS_1P^T$. The equivalence operations are therefore simultaneous row and column permutations, and simultaneous multiplication of a row and the respective column by $-1$. The signed permutation matrices $P$ for which $S=PSP^T$ form the full automorphism group of $S$, denoted by $\mathrm{Aut}(S)$. The identity matrix $I$ and its negative $-I$ are automorphisms of every $S$. The number of Seidel matrices of order $n$ up to equivalence is well-known; for an explicit formula see \cite{L}, \cite{MS}, \cite{R}. Similarly, an explicit formula to determine the number of self-complementary Seidel matrices (for which $S\sim -S$) can be found in \cite{NEWSCREF}. Here we are interested in efficiently generating and studying Seidel matrices from an algebraic perspective.

Let $S$ be a Seidel matrix of order $n\geq 2$ with spectrum $\Lambda(S)$. Since $S$ is symmetric, the elements of the multiset $\Lambda(S)$ are real numbers. Let $\lambda_{\min}\in\Lambda(S)$ denote the smallest eigenvalue of $S$. Assuming that the multiplicity of $\lambda_{\min}$ is $n-d$ (for some $d\leq n-1$), then $G:=-S/\lambda_{\min}+I$ is a positive-semidefinite matrix of rank $d$. The matrix $G$ can be viewed as a Gram matrix corresponding to an equiangular line system of $n>d$ lines in $\mathbb{R}^{d}$, where the pairwise inner product between any two distinct lines is $\pm1/\lambda_{\min}$, see \cite{LS}, \cite{vLS}. The value $\alpha:=-1/\lambda_{\min}$ is called the common angle between the lines, and is known to be the inverse of an odd integer once $n>2d$. This correspondence motivates the study of Seidel matrices and highlights the importance of their smallest eigenvalue.

The research carried out in this paper was motivated by the question of the existence of $29$ equiangular lines in $\mathbb{R}^{14}$ with common angle $1/5$, see \cite{GRWSN}, \cite{GRWSN2}. While we were not able to settle this problem, the computational results presented here can be viewed as a first minor step towards its resolution. The outline of this paper is as follows. In Section~\ref{sect2} we describe the graph-theoretical framework required for the computer representation of Seidel matrices. In Section~\ref{sect3} we study the characteristic polynomial of Seidel matrices of order $n\leq 13$, and we determine the number of cospectral Seidel matrices; the number of Seidel matrices with exactly $k$ distinct eigenvalues ($1\leq k\leq 13$); and the number of Seidel matrices with smallest eigenvalue $\lambda_{\min}\in\{-3,-5,-7\}$. Similar results are known for graphs, see \cite{graphspectra2}, \cite{graphspectra1}. In Section~\ref{sect4} we give a further look at Seidel matrices with exactly three distinct eigenvalues and generate and enumerate them up to $n\leq 23$. Several hypothetical Seidel matrices with integer spectrum were shown not to exist. This extends some earlier work done in \cite{GRWSN2}. Additionally, we observe that there exist Seidel matrices of order $18$ and $30$ which do not correspond to a regular graph, thus answering a recent question raised in \cite[Question~B]{GRWSN}. Again, similar classification results are known for graphs with few distinct eigenvalues, see \cite{VD}, \cite{vDS}, \cite{MzK}. In Sections~\ref{sect3} and \ref{sect4} we also recall several elementary matrix theoretical tools from \cite{HJ95}. In Section~\ref{sect5} we give a look at the equiangular lines problem from a computational perspective. Our computation reveals that the maximum number of equiangular lines in $\mathbb{R}^{12}$ with common angle $1/5$ is $20$. This resolves a question left open in \cite{GRWSN2}. In Section~\ref{sect6} we conclude the paper with the description of certain consistency checks we made during our computations for verification purposes. 

\section{Computer calculations of Seidel matrices}\label{sect2}
In this section we set up a computational framework for Seidel matrices. The standard way to deal with equivalence of various combinatorial objects is to represent them by graphs in a way so that the equivalence (and therefore automorphisms) of those are described by isomorphisms (and automorphisms) of graphs. Isomorphism of graphs is then decided via a canonical labeling algorithm. We use the C library \emph{nauty} (see \cite{MK}), which was designed to determine isomorphism of (colored) graphs.
\subsection{Representation of Seidel matrices}
Let $S$ be a Seidel matrix of order $n$. Then $A:=(J-I-S)/2$ (where $I$ is the identity matrix, and $J$ is the matrix with entries $1$ in every position) is the adjacency matrix of what is called the ambient graph $\Gamma(S)$, corresponding to $S$. However, since equivalent Seidel matrices can have nonisomorphic ambient graphs, this naive correspondence does not faithfully represent Seidel matrices and their symmetries. The set of nonisomorphic graphs corresponding in this way to a Seidel equivalence class is called the switching class of $S$. A result of Seidel implies that if $n$ is odd then every switching class contains a unique (up to isomorphism) graph having all vertex degrees even \cite{SEL}; there is no similar characterization when $n$ is even \cite{MS}.

In order to represent Seidel matirces faithfully, we assign to a Seidel matrix $S$ of order $n$ a two-colored graph on $3n$ vertices, denoted by $X(S)$. The vertex set of $X(S)$ is
\[V(X(S))=\{u_0,u_1,\hdots, u_{n-1}\}\cup\{v_{0}^{(0)},v_0^{(1)},v_1^{(0)},v_1^{(1)},\hdots,v_{n-1}^{(0)},v_{n-1}^{(1)}\}.\]
The vertices $u_i$, $i\in\{0,\hdots,n-1\}$ form color class $0$, and all other vertices form color class $1$. The edge set of $X(S)$ is
\begin{multline*}
E(X(S))=\{\{u_i,v_i^{(k)}\} \colon i\in\{0,\dots,n-1\}; k\in\{0,1\}\}\\
\cup\{\{v_i^{(k)},v_j^{(k)}\} \colon S_{ij}=1,k\in\{0,1\}\}\cup\{\{v_i^{(k)},v_j^{(1-k)}\}\colon S_{ij}=-1; k\in\{0,1\}\}.
\end{multline*}
 Two graphs, $X_1$ and $X_2$, which are colored using the same set of colors are isomorphic, denoted by $X_1 \cong X_2$ if there is a color preserving and incidence preserving bijection between their vertices. The following fundamental lemma ascertains that studying the graphs described above is exactly the same as studying Seidel matrices. See \cite[Section~3.3.2]{KO} and \cite{MKHAD} for results of similar flavor.

\begin{lemma}\label{l21}
Let $S_1$ and $S_2$ be Seidel matrices. Then $S_1\sim S_2$ if and only if $X(S_1)\cong X(S_2)$. Moreover $\mathrm{Aut}(S_1)$ and $\mathrm{Aut}(X(S_1))$ are isomorphic as groups.
\end{lemma}
\begin{proof}
It is easy to see that the equivalence operations on Seidel matrices $S$ correspond to simple relabeling operations on their graphs $X(S)$. Indeed, a row and column permutation of $S$ would permute the vertices $u_i, i\in\{0,\hdots, n-1\}$ within themselves, and multiplication by $-1$ row and column $j\in\{0,\hdots, n-1\}$ of $S$ would transpose the vertices $v_j^{(0)}$ and $v_j^{(1)}$ in $X(S)$. Therefore equivalence of Seidel matrices implies isomorphism of their graphs.

To see the converse direction, one should prove that from a two-colored, unlabeled graph $Y\cong X(S)$, the unknown Seidel matrix $S$ can be uniquely reconstructed, up to equivalence. This reconstruction procedure is done via assigning labels to the vertices of $Y$. Let us label the vertices belonging to the $n$-element color class as $\{y_0,\hdots, y_{n-1}\}$ (in any order); the vertex with label $y_i$ will represent the $i$th row (and column) of a symmetric matrix $T$ with $0$ diagonal. Additionally, we see that for every $i\in\{0,\hdots n-1\}$ vertex $y_i$ is adjacent to two other vertices, labeled as $z_i^{(0)}$ and $z_i^{(1)}$, which are themselves non adjacent and belong to the $2n$-element color class. Furthermore for indices $i\neq j\in\{0,\hdots , n-1\}$ either of the following two conditions hold: (a) $z_i^{(0)}$ is adjacent to $z_j^{(0)}$ and $z_i^{(1)}$ is adjacent to $z_j^{(1)}$; or (b) $z_i^{(0)}$ is adjacent to $z_j^{(1)}$ and $z_i^{(1)}$ is adjacent to $z_j^{(0)}$. From this particular labeling we can set $T_{ij}=1$ if condition (a) holds for the indices $i\neq j$, and we set $T_{ij}=-1$ otherwise. What remains to be seen is that the equivalance class of $T$ is invariant up to the chosen labeling of $Y$. This is indeed the case, as any relabeling of the vertices $y_i$ would correspond to row and column permutations of $T$, and the transposition of vertices $z_i^{(0)}$ and $z_i^{(1)}$ would change the element $T_{ik}$ to $-T_{ik}$, and $T_{ki}$ to $-T_{ki}$ for every $k\in\{0,\dots,n-1\}$. Since both of these operations leave $T$ invariant up to Seidel equivalence, and no other relabeling is possible, this proves that $T$ is unique up to equivalence.

Finally, the symmetry group $\mathrm{Aut}(S)$ acting on a Seidel matrix $S$ is generated by the row/column permutations $P\in \mathrm{Aut}(S)$; and by the row/column switching operations $D\in \mathrm{Aut}(S)$. These generators correspond to the generators of the symmetry group $\mathrm{Aut}(X(S))$ of the graph $X(S)$, where $p\in\mathrm{Aut}(X(S))$ is some permutation of the vertices $u_i$; and $d\in\mathrm{Aut}(X(S))$ is the transposition of the vertices $v_j^{(0)}$ and $v_j^{(1)}$ for some $j\in\{0,\hdots, n-1\}$. This correspondence is a group isomorphism.
\end{proof}

The automorphism group of a graph partitions its vertex set into vertex orbits. Two vertices, $u$ and $v$ are in the same vertex orbit, denoted by $u\sim v$, if there is an automorphism mapping $u$ onto $v$. If two vertices are in the same orbit, then they have necessarily the same color, thus the vertex orbits form a partition of the individual color classes. We will see shortly the significance of finding out whether a vertex forms a one-element orbit. This can be determined in certain cases by vertex invariants. A vertex invariant is a vertex function $f$ taking the same value on the elements of a vertex orbit. We assume that the range of $f$ is linearly ordered. Clearly, if $u$ and $v$ we are distinct vertices with $f(u)\neq f(v)$, then they cannot belong to the same vertex orbit. In particular, if a certain value is assumed only once, then the corresponding vertex forms a one-element orbit.

We note the following elementary, yet important property (see e.g.\ \cite{GODSILXXX}).

\begin{lemma}\label{lemma2}
Let $X$ be a graph. If two vertices $u$ and $v$ are in the same vertex orbit, then the subgraphs induced by $V(X)\setminus \{u\}$ and $V(X)\setminus \{v\}$, obtained by deleting $u$ and $v$, are isomorphic.
\end{lemma}

We observe that the orbit partition of the color $0$ vertices of $X(S)$ in turn induces a partition of the rows of $S$. We call this partition the row orbit partition of $S$. As a consequence of Lemma~\ref{lemma2} if $u_i\sim u_j$ are color $0$ vertices of $X(S)$ for some $i,j\in\{0,\dots, n-1\}$, then removing either of the $i$th row and column, or the $j$th row and column of $S$ would yield equivalent sub-Seidel matrices of order $n-1$.

The C library \emph{nauty} \cite{MK} performs canonical labeling, determines the automorphism group, and determines the vertex orbits under the action of the automorphism group of a colored graph. We note that this implementation of canonical labeling respects the order of the color classes. In particular, in any canonical labeling we encounter, labels $0,\dots,n-1$ will be assigned to vertices $u_i$ and labels $n,\dots,3n-1$ to vertices $v_i^{(k)}$ of $X(S)$, $i\in\{0,\hdots, n-1\}$, $k\in \{0,1\}$.

\subsection{Generation of Seidel matrices}\label{sss2}

We have generated the Seidel matrices of order $n\leq 13$ with a variant of canonical augmentation, see \cite{BR}, \cite[Section~4.2.3]{KO} and \cite{MKISO}.

Given a set $\mathcal{S}$ containing exactly one representative of every Seidel matrix of order $n$ up to equivalence, our goal is to generate and record in a set $\mathcal{T}$ exactly one representative of every Seidel matrix $\widehat{S}$ of order $n+1$, up to equivalence. We augment every starting-point Seidel matrix $S\in\mathcal{S}$ with a new row and column in every possible way, and record in a set $\mathcal{U}$ the augmented matrices $\widehat{S}$ up to equivalence. For a given matrix $\widehat{S}$ we construct the graph $X(\widehat{S})$, and call \emph{nauty} to determine a canonical labeling and the vertex orbits of the color $0$ vertices under the action of the automorphism group. Let $p\in V(X(\widehat{S}))$ be the vertex which got canonical label $0$. The strategy for rejecting potential isomorphic graphs is the following: on the one hand, if $u_n$, which is the most recently appended color $0$ vertex of $X(\widehat{S})$, and $p$ are in different orbits, then $\widehat{S}$ is immediately discarded. On the other hand, if $u_n\sim p$, then the canonical graph of $X(\widehat{S})$ is compared to the canonical graphs representing the elements of $\mathcal{U}$. The matrix $\widehat{S}$ is recorded and added to $\mathcal{U}$ if it has not been found so far, and discarded otherwise. Once all possible augmentations of $S$ were inspected, the contents of $\mathcal{U}$ are appended to $\mathcal{T}$, $\mathcal{U}$ is reset to the empty set, and the algorithm moves on to process the next element of $\mathcal{S}$.

\begin{lemma}
Given a set $\mathcal{S}$ containing exactly one representative of the equivalence classes of Seidel matrices of order $n$, the algorithm described above reports exactly one representative of every Seidel matrix of order $n+1$.
\end{lemma}
\begin{proof}
Consider a Seidel matrix $S$ of order $n+1$, and compute a canonical labeling of $X(S)$. This reveals the vertex $u_i$ (or more precisely, the vertex orbit containing $u_i$) getting canonical label $0$ ($i\in \{0,\hdots, n\})$. Thus removing the $i$th row and column of $S$ yields a Seidel matrix of order $n$ which in turn is represented by some starting-point matrix $S_0\in \mathcal{S}$, so that $S\sim\widehat{S_0}$ for some augmentation. Hence $S$, up to equivalence, is reported.

Moreover, if $\widehat{S_1}$ and $\widehat{S_2}$ were equivalent Seidel matrices of order $n+1$, resulting from augmenting the starting-point matrices $S_1$ and $S_2$, respectively, then necessarily $S_1=S_2$. This follows because these matrices were not discarded, and therefore the most recently appended color $0$ vertices to their graphs $X(\widehat{S_1})$ and $X(\widehat{S_2})$ were in the same vertex orbit as the vertex with canonical label $0$. Thus by Lemma~\ref{lemma2} removing the vertices with canonical label $0$ yield isomorphic graphs. This in turn implies that $X(S_1)\cong X(S_2)$ and hence $S_1\sim S_2$. By the properties of $\mathcal{S}$ this can only happen if $S_1=S_2$. This establishes that equivalent matrices to $S$ are all reported from the same starting point matrix. However, (the canonical graphs of) these matrices are compared (with the elements of $\mathcal{U}$) before they are recorded and those which have already been found were discarded during the search. Therefore the Seidel matrices are reported exactly once.
\end{proof}

A nice feature of this generation algorithm and the rejection strategy involved is that it is possible to augment the elements of $\mathcal{S}$ in parallel. Also, the number of Seidel matrices recorded in the set $\mathcal{T}$ does not depend on the number of starting-point Seidel matrices; it only depends on the number of inequivalent augmented matrices $\widehat{S}$, stemming from the same starting-point matrix $S$.

The general idea behind the isomorph rejection strategy is to assign to the equivalence classes of Seidel matrices $\widehat{S}$ of order $n+1$ a preferred row orbit (which in turn by Lemma~\ref{lemma2} assigns a canonical parent object $S$ of order $n$), and record a matrix during the search only if its most recently appended row belongs to this row orbit. In the algorithm described in the beginning of Section~\ref{sss2} this row orbit was given by the vertex orbit of $X(S)$, which contained the vertex with canonical label $0$. However, since performing a canonical labeling is computationally expensive, alternative strategies for selecting the preferred row orbit (and thus the canonical parent) should be employed for certain equivalence classes. To this end, we recall a technique based on vertex invariants.

Let $S$ be a Seidel matrix of order $n$, and let $f$ be a vertex invariant of the vertex orbits of $X(S)$ with ordering $\prec$ on its range, and consider its values taken on the color $0$ vertices $u_i\in V(X(S)), i\in\{0,\dots,n-1\}$. Let $U$ be the set of vertices with unique invariant value, that is $U:=\{u_i : i\in\{0,\hdots, n-1\} ; \forall j\in\{0,\hdots n-1\}, j\neq i : f(u_i)\neq f(u_j)\}$. Clearly, the members of $U$ form one-element vertex orbits and hence by Lemma~\ref{lemma2} for every $u_i\in U$ the subgraphs induced by $V(X(S))\setminus \{u_i, v_i^{(0)},v_i^{(1)}\}$ are unique (up to isomorphism). If $U$ is nonempty, then let the preferred vertex orbit be the one-element set $\{u : u \in U ; \forall x\in U f(u) \prec f(x)\}$, otherwise if $U$ is empty, then we fall back to the strategy based on canonical labeling. In that case, the preferred vertex orbit remains the one which contains the vertex with canonical label $0$.

The sensitivity of $f$, viz.\ the ratio of the number of vertex orbits it distinguishes to the total number of vertex orbits is greatly depends on the choice of $f$: simple vertex invariants are computationally cheap, yet they might not be able to recognize one-element vertex orbits. Indeed, the vertex degree $f(u):=\mathrm{degree}(u)$, the most natural choice for vertex invariants, is unfortunately constant on the color classes of $X(S)$ and consequently useless for our purposes. We have used the following vertex invariant for Seidel matrices of order $n\geq 3$:
\begin{multline*}
f(u_i):=|\{(j,k) : j,k\in\{0,\hdots, n-1\} ; j\neq i, k\neq i, j\neq k ; \text{the induced subgraph on}\\
 \{u_i,u_j,u_k, v_i^{(0)}, v_i^{(1)},v_j^{(0)}, v_j^{(1)},v_k^{(0)}, v_k^{(1)}\} \text{ is isomorphic to } X(J_3-I_3)\}|,
\end{multline*}
Equivalently, $f(u_i)$ counts the number of $3\times 3$ sub-Seidel matrices intersecting the $i$th row and column of $S$, which are equivalent to $J_3-I_3$.

During the classification process once we encounter an augmented matrix $\widehat{S}$, we compute the values of this vertex invariant first, and then check if there are vertices with unique invariant value. If this is the case, then we discard $\widehat{S}$ unless its most recently augmented vertex $u_n \in V(X(\widehat{S}))$ is the one with smallest unique invariant value. If there are no vertices with unique invariant value, then it might be more efficient to try another, more powerful invariant first (e.g., one based on the $4\times 4$ sub-Seidel matrices) before falling back to canonical labeling. If none of the invariants used were able to find a unique vertex (perhaps because there were no such vertices at all), then we can refine the coloring of the vertices $u_i$ of $X(\widehat{S})$ based on the values taken by $f$. This extra information, if implemented carefully, could greatly improve the efficiency of \emph{nauty}. Another possible improvement is to utilize $\mathrm{Aut}(S)$ in canonical augmentation, see \cite{BR}, \cite[Section~4.2.3]{KO}.

Our C implementation classified the $13\times 13$ Seidel matrices in about a day in a cluster of $16$ computers having $12$ logical processors each. The computed data takes up some 120GB of storage in graph6 format. A pseudocode describing this algorithm is available in the Appendix.

\section{Algebraic properties of Seidel matrices}\label{sect3}
Recall that the spectrum of a Seidel matrix $S$ of order $n$, denoted by $\Lambda(S)$, contains its eigenvalues, $\lambda_0$, $\hdots$, $\lambda_{n-1}$, which are in turn the roots of the characteristic polynomial
\[p_S(\lambda):=\mathrm{det}(\lambda I-S)=\lambda^n+c_{n-2}\lambda^{n-2}+\hdots+c_0.\]
We note that the characteristic polynomial has leading coefficient $1$ (it is monic), and therefore all of its roots are algebraic integers.

Two Seidel matrices $S_1$ and $S_2$ are cospectral, and $S_2$ is called the cospectral mate of $S_1$, if $\Lambda(S_1)=\Lambda(S_2)$. Equivalent Seidel matrices are cospectral, but starting from $n=8$ examples of inequivalent cospectral matrices appear. In this section we describe results regarding the characteristic polynomial of Seidel matrices of order $n\leq 13$. In particular, we tabulate the number of cospectral Seidel matrices, the distribution of Seidel matrices with exactly $k\in\{1,\dots,n\}$ distinct eigenvalues, and the number of Seidel matrices with smallest eigenvalue $\lambda_{\min}\in\{-3,-5,-7\}$. It turns out, that all of these properties can be treated via inspecting the characteristic polynomial. 

There are various algorithms available for computing the characteristic polynomial. We decided to implement a classical algorithm due to the Faddeev--LeVerrier \cite{FL}, and in addition for verification and benchmarking purposes a fraction-free Gaussian elimination algorithm \cite[Chapter~9]{GEDDES}. We found that for these small matrices Faddeev--LeVerrier performed better.

Table~\ref{table1} summarizes our findings regarding the number of cospectral Seidel matrices up to equivalence. In the columns the size, the number of Seidel matrices up to equivalence, the number of distinct characteristic polynomials, the number of Seidel matrices having a cospectral mate, and the largest number of pairwise inequivalent cospectral Seidel matrices are given. It appears that the fraction of Seidel matrices with a cospectral mate increases at first, but from $n=12$ starts to decrease. This might suggest that cospectral matrices in general are rare. A similar phenomenon was observed in connection with graph adjacency matrices \cite{graphspectra2}, \cite{graphspectra1}.

\tiny
\begin{table}[!h]
\[\begin{array}{rrrrr}
\hline
n & \text{\#Seidel matrices} & \text{\#char.\ polys} & \text{\#with mate} & \text{max.\ family}\\
\hline
 1 &           1 &           1 &          0 &   1\\
 2 &           1 &           1 &          0 &   1\\
 3 &           2 &           2 &          0 &   1\\
 4 &           3 &           3 &          0 &   1\\
 5 &           7 &           7 &          0 &   1\\
 6 &          16 &          16 &          0 &   1\\
 7 &          54 &          54 &          0 &   1\\
 8 &         243 &         235 &         15 &   3\\
 9 &        2038 &        1824 &        400 &   4\\
10 &       33120 &       28488 &       8340 &  12\\
11 &     1182004 &      925108 &     437484 &  25\\
12 &    87723296 &    71366612 &   28175661 &  83\\
13 & 12886193064 & 10746314335 & 3722801719 & 174\\
\hline
\end{array}\]
\caption{The number of cospectral Seidel matrices for $n\leq 13$}\label{table1}
\end{table}
\normalsize

The computation of the numbers given in Table~\ref{table1} was carried out in a similar fashion to what is described in detail in \cite{graphspectra1}. The main difficulty arose from the fact that the number of characteristic polynomials we initially encountered were too numerous to store and process at once. Therefore at first we split up the Seidel matrices of order $13$ according to the value of $|\mathrm{det}(S)|\ (\mathrm{mod}\ p)$ for some prime number $p$. We chose $p=761$ and $p=383$ in two independent rounds of computation. Clearly, if $S_1$ and $S_2$ have distinct absolute determinant modulo $p$, then their characteristic polynomials are distinct. Therefore it is possible to investigate these smaller sets independently. Computing and evaluating these characteristic polynomials took about a week.

Another application of computing the characteristic polynomial is to determine the number of Seidel matrices which only have few distinct eigenvalues. The following simple result was used in this regard (see also \cite[p.~54]{HJ95}). Henceforth we denote by $\mathrm{deg}(p)$ the degree of the polynomial $p$; by $p'$ the derivative of $p$; and by $\mathrm{GCD}(p,q)$ the greatest common divisor of the polynomials $p$ and $q$.
\begin{lemma}\label{l31}
Let $A$ be a matrix of order $n\geq 1$ with characteristic polynomial $p_A(\lambda)$. Then $A$ has exactly $n-\mathrm{deg}(\mathrm{GCD}(p_A(\lambda),p_A'(\lambda)))$ distinct eigenvalues.
\end{lemma}
\begin{proof}
Assume that $A$ is a matrix of order $n\geq 1$ with exactly $k\geq 1$ distinct eigenvalues $\lambda_0, \dots, \lambda_{k-1}$ of multiplicity $r_0, \hdots, r_{k-1}$, $r_i\geq 1$ for $i\in\{0,\hdots,k-1\}$. Then, we have $p_A(\lambda)=(\lambda-\lambda_0)^{r_0}\cdots(\lambda-\lambda_{k-1})^{r_{k-1}}$, and $p_A'(\lambda)=(\lambda-\lambda_0)^{r_0-1}\cdots(\lambda-\lambda_{k-1})^{r_{k-1}-1}q(x)$, where $q(\lambda_i)\neq 0$ for every $i\in\{0,\hdots,k-1\}$, and the result follows.
\end{proof}

During the computation of the characteristic polynomials, additionally, we recorded the frequency of the number $k:=n-\mathrm{deg}(\mathrm{GCD}(p_S(\lambda),p_S'(\lambda)))$. Our results are shown in Table~\ref{table3x}.

\tiny
\begin{table}[ht]
\[\begin{array}{r|rrrrrrrrrrrrr}
{}_k\mkern-6mu\setminus\mkern-6mu{}^n       & 1 & 2 & 3 & 4 & 5 & 6 &  7 &  8 &   9 &    10 &     11 &       12 &          13\\
\hline
 1        & 1 & 0 & 0 & 0 & 0 & 0 &  0 &  0 &   0 &     0 &      0 &        0 &           0\\
 2        &   & 1 & 2 & 2 & 2 & 3 &  2 &  2 &   2 &     3 &      2 &        2 &           2\\
 3        &   &   & 0 & 0 & 1 & 2 &  0 &  2 &   3 &     4 &      0 &       10 &           1\\
 4        &   &   &   & 1 & 4 & 5 & 16 & 17 &  20 &    37 &     38 &       66 &          64\\
 5        &   &   &   &   & 0 & 2 &  8 & 20 &  55 &    56 &     92 &      166 &         184\\
 6        &   &   &   &   &   & 4 & 20 & 64 & 188 &   406 &    830 &     1487 &        2362\\ 
 7        &   &   &   &   &   &   &  8 & 46 & 218 &   696 &   2110 &     4620 &       12665\\
 8        &   &   &   &   &   &   &    & 92 & 652 &  3507 &  14336 &    47515 &      135038\\
 9        &   &   &   &   &   &   &    &    & 900 &  5960 &  41276 &   174522 &      721408\\
10        &   &   &   &   &   &   &    &    &     & 22451 & 217090 &  1848989 &    11633134\\
11        &   &   &   &   &   &   &    &    &     &       & 906230 &  7082344 &    70565315\\
12        &   &   &   &   &   &   &    &    &     &       &        & 78563575 &   817640532\\
13        &   &   &   &   &   &   &    &    &     &       &        &          & 11985482359\\
\end{array}\]
\caption{The number of Seidel matrices of order $n$ having exactly $k$ distinct eigenvalues}\label{table3x}
\end{table}
\normalsize

A few comments are in order. Since Seidel matrices have $0$ trace, any Seidel matrix of order $n\geq 2$ has both a negative and a positive eigenvalue. The Seidel matrices with exactly two distinct eigenvalues are called regular two-graphs in the literature, see \cite{GMS}, \cite{MSP}, \cite{T}, \cite{TTH}, and correspond to several sporadic examples of large set of equiangular lines. It is clear that the Seidel matrices $\pm(J-I)$ have exactly two distinct eigenvalues, but examples beyond these are somewhat rare. The Seidel matrices with exactly three distinct eigenvalues were studied first in \cite{GRWSN2} where they were enumerated up to $n\leq 12$, and most recently in \cite{GRWSN}, where their structural properties were investigated. For example, it is known that they do not exist for prime orders $p\equiv 3\ (\mathrm{mod}\ 4)$, see \cite[Theorem~5.9]{GRWSN2}. Close inspection of Table~\ref{table3x} reveals that the number of Seidel matrices with exactly $4$ distinct eigenvalues is not strictly increasing. This might suggest that they are rare enough to be interesting.

A third application of the characteristic polynomial is that it characterizes positive semidefinite matrices. Recall that a real symmetric matrix is called positive semidefinite, if all of its eigenvalues are nonnegative. The following is a slight reformulation of \cite[Corollary~7.2.4]{HJ95}.

\begin{lemma}\label{l32}
A real symmetric matrix $A=A^T$ of order $n$ is positive semidefinite if and only if all the coefficients of its characteristic polynomial $p_A(\lambda)=\lambda^n+c_{n-1}\lambda^{n-1}+\hdots+c_0$ alternate in sign, that is, they satisfy $(-1)^{n-i}c_i\geq 0$ for every $i\in\{0,\dots,n-1\}$.
\end{lemma}
\begin{proof}
On the one hand, it is easy to see that if the coefficients of $p_A(\lambda)$ alternate in sign, then such a polynomial cannot have a negative root. Therefore, since $A=A^T$, all other roots are nonnegative, and hence $A$ is positive semidefinite. On the other hand, if $A$ is positive semidefinite with $m\leq n$ positive eigenvalues $\lambda_0, \dots, \lambda_{m-1}$, then an inductive argument shows that the signs of the coefficients of $\prod_{i=0}^{m-1}(\lambda-\lambda_i)$ alternate strictly; multiplying by $\lambda^{n-m}$ gives $p_A(\lambda)$.
\end{proof}
Here we are interested in the number of Seidel matrices $S$ with smallest eigenvalue $\lambda_{\min}\geq x\in\{-3,-5,-7\}$. This question can be decided via Lemma~\ref{l32} by considering the matrix $S-xI$ and determining whether it is positive semidefinite. While it is possible to use a complete classification of Seidel matrices of order $n$, and then check this property for every matrices, in certain cases, for example, when $x=-3$, it is faster to generate them as described in Section~\ref{sect2}, but discarding along the way the Seidel matrices with $\lambda_{\min}< -3$. The reason for this is interlacing: if a Seidel matrix has smallest eigenvalue $\lambda_{\min}$, then all of its sub-Seidel matrices have smallest eigenvalue at least $\lambda_{\min}$. We state this result as follows, see \cite[Theorem~4.3.28]{HJ95}.
\begin{theorem}\label{tint}
Let $A=A^T$ be a real symmetric matrix of order $n$, partitioned as $\bigl[\begin{smallmatrix}B & C\\ C^T & D\end{smallmatrix}\bigr]$ where $B=B^T$ is of order $m\leq n$. Let the eigenvalues of $A$ be $\lambda_0\leq \dots\leq \lambda_{n-1}$, and let the eigenvalues of $B$ be $\mu_0\leq\dots\leq \mu_{m-1}$. Then $\lambda_i\leq \mu_i\leq \lambda_{i+n-m}$, $i\in\{0,\dots,m-1\}$.
\end{theorem}
The proof of Theorem~\ref{tint}, along with the discussion of the case of equality, can be found in \cite[p.~246]{HJ95}. See also \cite{INT} for several applications of eigenvalue interlacing in graph theory. We remark that if $\lambda$ is an eigenvalue of $S$ of order $n$, then by the Ger\v{s}gorin circle theorem \cite[Theorem~6.1.1]{HJ95}, we have $|\lambda|\leq n-1$, and equality holds for the Seidel matrices $\pm(J-I)$.

Our findings are tabulated in Table~\ref{table3v2}. The focus on negative odd eigenvalues was motivated by the fact that large equiangular line systems correspond to Seidel matrices with such eigenvalues \cite{LS}.

\tiny
\begin{table}[ht]
\begin{tabular}{rrrrrrr}
\hline
$n$ & $\lambda_{\min}\geq-3$ & $\lambda_{\min}=-3$ & $\lambda_{\min}\geq-5$ & $\lambda_{\min}= -5$ & $\lambda_{\min}\geq -7$ & $\lambda_{\min} = -7$\\
\hline
 3 &     2 &    0 &         2 &        0 &           2 &       0\\
 4 &     3 &    1 &         3 &        0 &           3 &       0\\
 5 &     5 &    1 &         7 &        0 &           7 &       0\\
 6 &     9 &    4 &        16 &        1 &          16 &       0\\
 7 &    16 &    9 &        51 &        2 &          54 &       0\\
 8 &    25 &   23 &       215 &        8 &         243 &       1\\
 9 &    40 &   38 &      1601 &       33 &        2033 &       2\\
10 &    58 &   56 &     21249 &      306 &       33027 &      10\\
11 &    75 &   73 &    511275 &     6727 &     1177470 &      78\\
12 &    96 &   94 &  19032270 &   219754 &    87080911 &    1362\\
13 &   108 &  106 & 965697139 & 11295930 & 12660901908 &   55572\\
\hline
\end{tabular}
\caption{The number of Seidel matrices with eigenvalues $-3$, $-5$, and $-7$}\label{table3v2}
\end{table}
\normalsize

Inspection of Table~\ref{table3v2} reveals that while Seidel matrices with $\lambda_{\min}\geq-3$ or $\lambda_{\min}\leq-7$ are very rare, there is an abundance of Seidel matrices with $\lambda_{\min}\geq -5$ for $n\leq 13$. Based on the growth rate of these numbers, we rather conservatively estimate that there are some $3\times 10^{10}$ Seidel matrices of order $n=15$ with $\lambda_{\min}=-5$.

The Seidel matrices with these properties will serve as starting-point matrices for the search described in the next section where we study Seidel matrices with exactly three distinct eigenvalues.

\section{Seidel matrices with exactly three distinct eigenvalues}\label{sect4}
Since regular two-graphs are somewhat rare, one might wonder what other combinatorial objects could be useful to the construction of large set of equiangular lines. Inspection of these large sets reveals that most of them correspond to a Seidel matrix with exactly (two or) three distinct eigenvalues. This observation motivated the search initiated in \cite{GRWSN2}, where based on the full classification of Seidel matrices of order $n\leq 12$, enumeration of Seidel matrices with exactly three distinct eigenvalues was given. Here we extend that work slightly further, and enumerate all such Seidel matrices for $n\leq 23$. The summary of our results is given in Table~\ref{tableno3}.

\tiny
\begin{table}[ht]
\begin{tabular}{c|cccccccccccccccccccccc}
\hline
$n$& 3 & 4 & 5 & 6 & 7 & 8 & 9 & 10 & 11 & 12 & 13 & 14 & 15 & 16 & 17 & 18 & 19 & 20 & 21 & 22 & 23 & 24\\ 
\# & 0 & 0 & 1 & 2 & 0 & 2 & 3 & 4  &  0 & 10 &  1 &  2 &  6 &  4 &  1 & 12 &  0 & 30 &  6 &  2 &  0 & $\geq 20$\\
\hline
\end{tabular}
\caption{The number of Seidel matrices with exactly three distinct eigenvalues}\label{tableno3}
\end{table}
\normalsize

There are several general methods to construct Seidel matrices with exactly three distinct eigenvalues. The basic result is that if the ambient graph $\Gamma(S)$ is a connected regular graph with adjacency matrix $A$, where $A$ has exactly $r\in\{3,4\}$ distinct eigenvalues, then the Seidel matrix $S:=J-I-2A$ has at most $r$ distinct eigenvalues. Other, more involved constructions based on Seidel matrices with exactly two distinct eigenvalues are described in \cite[Proposition~5.11]{GRWSN2}, \cite[Lemma~5.12]{GRWSN2}, and \cite[Theorem~5.16]{GRWSN2}. Further sporadic examples are known \cite{TTH}. In the following we recall the result from \cite[Proposition~5.11]{GRWSN2} to demonstrate that such matrices exist for every composite order $n>4$. We use the short-hand notation $[\lambda]^a \in \Lambda(S)$ for $a\geq 1$ to indicate that the eigenvalue $\lambda$ has multiplicity at least $a$; the notation $[\lambda]^0\in\Lambda(S)$ carries no information about $\lambda$ and should be ignored. We denote by $A\otimes B$ the Kronecker product of the matrices $A$ and $B$.
\begin{lemma}\label{l41a}
Let $S$ be a Seidel matrix of order $b\geq 3$ with spectrum $\Lambda(S)=\{[\lambda_0]^{b-c},[\lambda_1]^c\}$ for some $c\geq 1$ with $\lambda_0,\lambda_1\neq1$. Then for $a\geq 1$, $S':=J_a\otimes (S-I_b)+I_{ab}$ has spectrum $\Lambda(S')=\{[a(\lambda_0-1)+1]^{b-c}, [1]^{(a-1)b}, [a(\lambda_1-1)+1]^{c}\}$.
\end{lemma}
\begin{proof}
Since $S'$ is a Seidel matrix, the claim follows immediately after noting that $\Lambda(J_a)=\{[0]^{a-1},[a]^1\}$.
\end{proof}
For $a\geq2,b\geq 3$ we use the notation $K(a,b):=J_a\otimes (J_b-2I_b)+I_{ab}$. By Lemma~\ref{l41a} $\Lambda(K(a,b))=\{[-2a+1]^{b-1},[1]^{ab-b},[ab-2a+1]^1\}$, and therefore the matrices $K(a,b)$ are examples of Seidel matrices with exactly three distinct eigenvalues. To compile a table listing all the potential spectrum of such matrices, we derive a simple property first.
\begin{lemma}\label{l41}
Let $S$ be a Seidel matrix of order $n=a+b+c$ with exactly three distinct eigenvalues given by $\Lambda(S)=\{[\lambda]^a,[\mu]^b,[\nu]^c\}$. Then necessarily
\[a=\frac{n(n-1+\mu\nu)}{(\lambda-\mu)(\lambda-\nu)},\qquad b=\frac{n(n-1+\lambda\nu)}{(\mu-\lambda)(\mu-\nu)},\qquad c=\frac{n(n-1+\lambda\mu)}{(\nu-\lambda)(\nu-\mu)}.\]
In particular, the quantities above are positive integers.
\end{lemma}
\begin{proof}
This follows from the system of equations $\mathrm{tr}(S)=a\lambda+b\mu+c\nu=0$, $\mathrm{tr}(S^2)=a\lambda^2+b\mu^2+c\nu^2=n(n-1)$, and $a+b+c=n$.
\end{proof}
A hypothetical spectrum $\Lambda(S)=\{[\lambda]^a,[\mu]^b,[\nu]^c\}$ is called feasible, if the integrality conditions of Lemma~\ref{l41} are met. There are further necessary conditions on the spectrum, some are general for all Seidel matrices, while others only apply to the three distinct eigenvalue case. We list the most important results below:
\begin{itemize}
\item The determinant of any Seidel matrix must satisfy $\mathrm{det}(S)\equiv 1-n\ (\mathrm{mod}\ 4)$, see \cite[Corollary~3.6]{GRWSN2}. This condition eliminates e.g.\ the case $\{[-\sqrt3]^1,[0]^1,[\sqrt3]^1\}$.
\item The multiplicity of any even eigenvalue must be $1$, see \cite[Theorem~2.2]{GRWSN2}. This eliminates $\{[-7]^1,[-1]^3,[2]^5\}$.
\item The following quantity $(n-1)(\lambda+\mu+\nu)+\lambda\mu\nu-n^2-n-2$ must be doubly even for Seidel matrices with exactly three distinct eigenvalues \cite[Corollary~5.3]{GRWSN2}. This eliminates $\{[1]^{8}, [-2+3\sqrt{3}]^3,[-2-3\sqrt{3}]^3\}$.
\item Finally, we note that a Seidel matrix with exactly three distinct eigenvalues necessarily has an integer eigenvalue \cite[Corollary~5.5]{GRWSN2}. This eliminates potential cases with eigenvalues having a cubic minimal polynomial. If there are any quadratic irrational eigenvalues, then they come in conjugate pairs having equal multiplicity, and their minimal polynomial must be monic. This eliminates $\{[3]^8,[-4+\sqrt{21}/3]^3,[-4-\sqrt{21}/3]^3\}$.
\end{itemize}
These necessary conditions, along with the Ger\v{s}gorin bound $|\lambda|\leq n-1$ allow us to compile a table listing the potential spectrum of hypothetical Seidel matrices. In Tables~\ref{table9x} and \ref{table10x} we do this for each $n\leq 24$, where the spectra satisfying all these necessary conditions is displayed, along with the number of inequivalent matrices found, with the convention in Table~\ref{table9x} that $\lambda < 0 \leq \mu\leq \nu$. Therefore if $S$ and its complement $-S$ have distinct spectra, then only one of those two cases is indicated.
\tiny
\renewcommand{\arraystretch}{1.2}
\begin{table}[ht]
\[\begin{array}{rrrrcc|rrrrcc}
\hline
n & [\lambda]^a & [\mu]^b & [\nu]^c & \# & \text{Remark} & n & [\lambda]^a & [\mu]^b & [\nu]^c & \# & \text{Remark}\\
\hline
 6 & [-3]^{2} & [1]^{3} & [3]^{1}   &   1    & K(2,3) & 20 & [-9]^{3} & [1]^{16} & [11]^{1}    & 1 & K(5,4) \\
 8 & [-3]^{3} & [1]^{4} & [5]^{1}   &   1    & K(2,4) & 20 & [-7]^{5} & [1]^{10} & [5]^{5}     &  4 & \{[-3]^5,[3]^5\} \\
 9 & [-5]^{2} & [1]^{6} & [4]^{1}   &   1    & K(3,3) & 20 & [-7]^{4} & [1]^{15} & [13]^{1}    & 1 & K(4,5) \\
 9 & [-3]^{4} & [0]^{1} & [3]^{4}   &   1    & \{[-3]^5,[3]^5\} & 20 & [-5]^{8} & [1]^{5} & [5]^{7} & 8 & \{[-5]^{13},[5]^{13}\} \\
10 & [-3]^{4} & [1]^{5} & [7]^{1}   &   1    & K(2,5) & 20 & [-3]^{9} & [1]^{10} & [17]^{1}    & 1 & K(2,10) \\
12 & [-7]^{2} & [1]^{9} & [5]^{1}   &   1    & K(4,3) & 21 & [-13]^{2} & [1]^{18} & [8]^{1}    &  1 & K(7,3) \\
12 & [-5]^{3} & [1]^{8} & [7]^{1}   &   1    & K(3,4) & 21 & [-5]^{6} & [1]^{14} & [16]^{1}    & 1 & K(3,7) \\
12 & [-3]^{6} & [1]^{3} & [5]^{3}   &   1    & \{[-3]^{10},[5]^6\} & 21 & [-3]^{14} & [0]^{1} & [7]^{6}     & 1 &  \\
12 & [-3]^{5} & [1]^{6} & [9]^{1}   &   1    & K(2,6) & 21 & [-3]^{14} & [5]^{6} & [12]^{1}    & 0 &  \\
14 & [-3]^{6} & [1]^{7} & [11]^{1}  &   1    & K(2,7) & 22 & [-3]^{10} & [1]^{11} & [19]^{1}   & 1 & K(2,11) \\
15 & [-9]^{2} & [1]^{12} & [6]^{1}  &   1    & K(5,3) & 24 & [-15]^{2} & [1]^{21} & [9]^{1}    & 1 & K(8,3) \\
15 & [-5]^{4} & [1]^{10} & [10]^{1} &   1    & K(3,5) & 24 & [-11]^{3} & [1]^{20} & [13]^{1}   & 1 & K(6,4) \\
15 & [-3]^{9} & [2]^{1} & [5]^{5}   &   1    & \{[-3]^{10},[5]^6\} & 24 & [-7]^{7} & [1]^{9} & [5]^{8}      & ? &  \\
16 & [-7]^{3} & [1]^{12} & [9]^{1}  &   1    & K(4,4) & 24 & [-7]^{6} & [1]^{15} & [9]^{3}     & \geq3 & \text{incomplete search} \\
16 & [-3]^{8} & [1]^{6} & [9]^{2}   &   0     &        & 24 & [-7]^{5} & [1]^{18} & [17]^{1}    & 1 & K(4,6) \\
16 & [-3]^{7} & [1]^{8} & [13]^{1}  &   1    & K(2,8) & 24 & [-5]^{10} & [1]^{8} & [7]^{6}     & ? &  \\
18 & [-11]^{2} & [1]^{15} & [7]^{1} &   1    & K(6,3) & 24 & [-5]^{8} & [1]^{14} & [13]^{2}    & 0 & \\
18 & [-9]^{3} & [1]^{9} & [3]^{6}   &   0     &        & 24 & [-5]^{7} & [1]^{16} & [19]^{1}    & 1  & K(3,8) \\
18 & [-5]^{6} & [1]^{9} & [7]^{3}   &   1    &        & 24 & [-5]^{11} & [3]^{9} & [7]^{4}     & ? &   \\
18 & [-5]^{5} & [1]^{12} & [13]^{1} &   1    & K(3,6) & 24 & [-3]^{16} & [1]^{3} & [9]^{5}     & 0 &   \\
18 & [-3]^{8} & [1]^{9} & [15]^{1}  &   1    & K(2,9) & 24 & [-3]^{11} & [1]^{12} & [21]^{1}   & 1 & K(2,12)  \\
18 & [-3]^{11} & [3]^{5} & [9]^{2}  &   1    &        & 24 & [-3]^{16} & [3]^{5} & [11]^{3}    & 0 &  \\
   &          &         &           &  &  & 24 & [-3]^{17} & [5]^{3} & [9]^{4} & 1 &  \{[-3]^{21}, [9]^7\}\\
\hline
\end{array}\]
\caption{Seidel matrices with exactly three distinct integer eigenvalues}\label{table9x}
\end{table}
\renewcommand{\arraystretch}{1}
\normalsize

Before describing the search in detail, we comment on Tables~\ref{table9x} and \ref{table10x}. The data in Table~\ref{table9x} shows that within the range of $n\leq 24$ most examples come from the matrices $K(a,b)$. If a certain case can be obtained by application of \cite[Lemma~5.12]{GRWSN2} or \cite[Theorem~5.16]{GRWSN2}, then the spectrum of the corresponding regular two-graph (or its complement) is noted. The cases where the smallest (or largest) eigenvalue was $-3$ (or $3$, respectively) could be very easily classified (cf.\ Table~\ref{table3v2}), while those with smallest eigenvalue $-5$ required more efforts. The cases with a high eigenvalue multiplicity could be more easily treated than those where these multiplicities are approximately the same. In particular, addressing the existence of Seidel matrices with spectrum $\{[-7]^{7}, [1]^9, [5]^8\}$ is out of reach with these techniques, as the pruning conditions (see the next paragraphs) in these cases are very weak, and a search would essentially be as difficult as the complete classification of Seidel matrices of order $n\geq 14$.
\tiny
\renewcommand{\arraystretch}{1.2}
\begin{table}[ht]
\[\begin{array}{rrrrcc}
\hline
n & [\lambda]^a & [\mu]^b & [\nu]^c & \# & \text{Remark}\\
\hline
 5 & [0]^{1} & [\sqrt{5}]^{2} & [-\sqrt{5}]^{2}               & 1 & \text{From $\{[-\sqrt{5}]^3,[\sqrt{5}]^3\}$ via \cite[Lemma~5.12]{GRWSN2}} \\
8 & [1]^{4} & [-1 + 2 \sqrt{3}]^{2} & [-1 - 2 \sqrt{3}]^{2}   & 0 &  \\
10 & [3]^{4} & [-2 + \sqrt{5}]^{3} & [-2 - \sqrt{5}]^{3}      & 1 &  \text{Complement of the sporadic matrix $\mathfrak{S}$, see \cite{GRWSN}}\\
12 & [1]^{6} & [-1 + 2 \sqrt{5}]^{3} & [-1 - 2 \sqrt{5}]^{3}  & 1 & \text{From $\{[-\sqrt{5}]^3,[\sqrt{5}]^3\}$ via Lemma~\ref{l41a}} \\
13 & [0]^{1} & [\sqrt{13}]^{6} & [-\sqrt{13}]^{6}             & 1 & \text{From $\{[-\sqrt{13}]^7,[\sqrt{13}]^7\}$ via \cite[Lemma~5.12]{GRWSN2}} \\
16 & [1]^{8} & [-1 + 2 \sqrt{7}]^{4} & [-1 - 2 \sqrt{7}]^{4}  & 0 & \\
16 & [1]^{12} & [-3 + 4 \sqrt{3}]^{2} & [-3 - 4 \sqrt{3}]^{2} & 0 & \\
16 & [3]^{8} & [-3 + 2 \sqrt{3}]^{4} & [-3 - 2 \sqrt{3}]^{4}  & 0 &  \\
17 & [0]^{1} & [\sqrt{17}]^{8} & [-\sqrt{17}]^{8}             & 1 & \text{From $\{[-\sqrt{17}]^9,[\sqrt{17}]^9\}$ via \cite[Lemma~5.12]{GRWSN2}} \\
18 & [1]^{12} & [-2 + 3 \sqrt{5}]^{3} & [-2 - 3 \sqrt{5}]^{3} & 1 & \text{From $\{[-\sqrt{5}]^3,[\sqrt{5}]^3\}$ via Lemma~\ref{l41a}} \\
20 & [3]^{10} & [-3 + 2 \sqrt{5}]^{5} & [-3 - 2 \sqrt{5}]^{5} & 0 & \\
21 & [0]^{1} & [\sqrt{21}]^{10} & [-\sqrt{21}]^{10} & 0 & \text{There is no $\{[-\sqrt{21}]^{11},[\sqrt{21}]^{11}\}$}\\
24 & [1]^{12} & [-1 + 2 \sqrt{11}]^{6} & [-1 - 2 \sqrt{11}]^{6} & ? & \\
24 & [1]^{18} & [-3 + 4 \sqrt{5}]^{3} & [-3 - 4 \sqrt{5}]^{3} & 1 & \text{From $\{[-\sqrt{5}]^3,[\sqrt{5}]^3\}$ via Lemma~\ref{l41a}}\\
24 & [1]^{20} & [-5 + 6 \sqrt{3}]^{2} & [-5 - 6 \sqrt{3}]^{2} & 0 & \\
24 & [3]^{12} & [-3 + 2 \sqrt{7}]^{6} & [-3 - 2 \sqrt{7}]^{6} & 0 & \\
\hline
\end{array}\]
\caption{Seidel matrices with quadratic eigenvalues ($\lambda\geq 0$ is an integer)}\label{table10x}
\end{table}
\renewcommand{\arraystretch}{1}
\normalsize

In Table~\ref{table10x} cases with quadratic eigenvalues are listed. We observe that we have examples of order $n\equiv 1\ (\mathrm{mod}\ 4)$ with spectrum $\{[0]^1, [\sqrt{n}]^{(n-1)/2}, [-\sqrt{n}]^{(n-1)/2}\}$, coming from the so-called conference two-graphs with spectrum $\{[-\sqrt{n}]^{(n+1)/2}, [\sqrt{n}]^{(n+1)/2}\}$ (this is an instance of \cite[Lemma~5.12]{GRWSN2}). Further examples can be obtained by combining these with Lemma~\ref{l41a}. We also observe that some of the non-existent cases, e.g.\ $\{[1]^4, [-1+2\sqrt{3}]^2, [-1-2\sqrt{3}]^2\}$, would come from the analogous, yet non-existent regular two-graphs with spectrum $\{[-\sqrt{3}]^2, [\sqrt{3}]^2\}$.

Having compiled the tables with the potential spectra, the search for a Seidel matrix $S$ of order $n$ with spectrum $\Lambda(S)=\{[\lambda]^a,[\mu]^b,[\nu]^c\}$ was carried out in the same way to what is described in Section~\ref{sect2}. However, before accepting a Seidel matrix of order $m\leq n$ encountered during the search, two further tests were performed. The principles of these techniques are well-known, see \cite{CD}, \cite{MSP}. Here we describe them for completeness, and to point out some implementation details.

First it was tested if the eigenvalues $\theta_i$ of the generated matrix satisfy $\min\{\lambda,\mu,\nu\} \leq \theta_i \leq \max\{\lambda,\mu,\nu\}$ for all $i\in\{0,\dots,m-1\}$. This condition follows from Theorem~\ref{tint}. While this can be checked by two positive semidefiniteness tests by computing the characteristic polynomials based on Lemma~\ref{l32}, in practice it is much more efficient to test instead the slightly weaker condition $\min\{\lambda,\mu,\nu\}-\varepsilon < \theta_i < \max\{\lambda,\mu,\nu\}+\varepsilon$ for some $\varepsilon>0$, say $\varepsilon=0.01$. This requires a positive definiteness test, which can be done by computing $m$ determinants. Another advantage of this weaker condition is that it can be tested by using integer arithmetic only, thus eliminates the inconvenience of dealing with irrational eigenvalues. Recall that a real symmetric matrix is called positive definite, if all of its eigenvalues are positive.
\begin{proposition}[Sylvester's criterion]
The real symmetric matrix $A=A^T$ is positive definite if and only if every leading principal minor of $A$ is positive.
\end{proposition}
\begin{proof}
See \cite[Corollary~7.1.5]{HJ95} and \cite[Theorem~7.2.5]{HJ95}.
\end{proof}
Thus, during the search once we considered a Seidel matrix of order $m\leq n$, we tested if both $S-(\min\{\lambda,\mu,\nu\}-\varepsilon)I$ and $-S+(\max\{\lambda,\mu,\nu\}+\varepsilon)I$ are positive definite.

The second test was to check if $\lambda$, $\mu$, and $\nu$ were eigenvalues of sufficiently high multiplicity of the encountered matrices themselves. Again, it follows from Theorem~\ref{tint} that if $\lambda$ is an $a$-fold eigenvalue of $S$ of order $n$ ($1\leq a\leq n$), then $\lambda$ is an at least ($a-1$)-fold eigenvalue of any sub-Seidel matrix of order $n-1$. By induction, it follows that any sub-Seidel matrix of order $m:=n-(a-1)$ must have an eigenvalue $\lambda$. Deciding the multiplicity of a given eigenvalue $\lambda$ is a rank computation, which can be performed in integer arithmetic as long as $\lambda$ is an integer.

\begin{lemma}\label{l44}
Let $A=A^T$ be a square matrix of order $n$ and let $\lambda\in\Lambda(A)$ with multiplicity $m\geq1$. Then $m=n-\mathrm{rank}(A-\lambda I)$. Moreover, if $\lambda\not\in\Lambda(A)$, then $\mathrm{rank}(A-\lambda I)=n$.
\end{lemma}
\begin{proof}
Let $\lambda_0,\dots,\lambda_{n-m-1},\lambda_{n-m}=\lambda,\dots,\lambda_{n-1}=\lambda$ denote the eigenvalues of $A$. Since the matrix $A$ is normal, it is unitary diagonalizable by the spectral theorem. Therefore we can write $A= U\mathrm{diag}(\lambda_0,\hdots,\lambda_{n-1})U^{\ast}$ with some unitary matrix $U$, and in turn $A-\lambda I = (U\mathrm{diag}(\lambda_0-\lambda,\hdots,\lambda_{n-m-1}-\lambda,1,\hdots, 1))\mathrm{diag}(1,\hdots,1,0,\hdots, 0)U^\ast$, where the diagonal matrix on the right hand side has exactly $n-m$ nonzero entries. By part (f) of \cite[Theorem~0.4.6]{HJ95} this implies that $\mathrm{rank}(A-\lambda I)= n-m$.
\end{proof}

In case of Seidel matrices with quadratic eigenvalues, Lemma~\ref{l44} can only be used (within the framework of integer arithmetic) for the integer eigenvalue. For quadratic eigenvalues, an analogous condition holds based on the characteristic polynomial.
\begin{lemma}\label{l45}
Let $A=A^T$ be a square matrix of order $n$ and let $\lambda_0\in\Lambda(A)$ with mulitplicity $m\geq 1$. Let $p_A(\lambda)$ be the characteristic polynomial of $A$. Then $m=\mathrm{deg}(\mathrm{GCD}(p_A(\lambda),(\lambda-\lambda_0)^n)$. Moreover, if $\lambda_0\notin \Lambda(A)$ then $\mathrm{deg}(\mathrm{GCD}(p_A(\lambda),(\lambda-\lambda_0)^n)=0$.
\end{lemma}
\begin{proof}
This is immediate, as the highest exponent $k$ such that $p_A(\lambda)$ is divisible by $(\lambda-\lambda_0)^k$, is exactly $m$.
\end{proof}
The following immediate corollary of Lemma~\ref{l45} is powerful enough to deal with the quadratic case within the framework of integer arithmetic.
\begin{corollary}\label{c46}
Let $A=A^T$ be a square matrix of order $n$ and let $\mu, \nu\in\Lambda(A)$ be distinct conjugate algebraic integers of degree $2$ with mulitplicity $m\geq 1$ each. Let $p_A(\lambda)$ be the characteristic polynomial of $A$. Then $m=\mathrm{deg}(\mathrm{GCD}(p_A(\lambda),(\lambda^2-(\mu+\nu)\lambda+\mu\nu)^{\left\lfloor n/2\right\rfloor})$. Moreover, if $\mu\notin \Lambda(A)$ then $\nu\notin\Lambda(A)$ and $\mathrm{deg}(\mathrm{GCD}(p_A(\lambda),(\lambda^2-(\mu+\nu)\lambda+\mu\nu)^{\left\lfloor n/2\right\rfloor})=0$.
\end{corollary}
We remark here that Lemma~\ref{l45} and Corollary~\ref{c46} can be reformulated in terms of higher order derivatives of $p_A(\lambda)$.

While the positive definiteness tests allowed certain matrices with not necessarily smallest eigenvalue $\min\{\lambda,\mu,\nu\}$ to be considered and augmented during the search, once we have reached matrices of size $n$, all of these were eliminated by repeated applications of Lemma~\ref{l44} and Corollary~\ref{c46}. Thus at the end of the search we were either able to generate all Seidel matrices of the given spectrum up to equivalence, or conclude that no such matrix exists.

In a recent paper Seidel matrices with exactly three distinct eigenvalues were studied \cite{GRWSN}, and it was proved that a large family of these have a regular graph in their switching class. This correspondence is useful because graphs are simpler objects to study, and known nonexistence results on regular graphs yield nonexistence results of Seidel matrices. We conclude this section by two results. The first highlights one of our nonexistence results, which does not seem to follow from any of the recent theoretical considerations \cite{GRWSN}, \cite{GRWSN2}.
\begin{theorem}
Seidel matrices of order $18$ with spectrum $\{[-9]^3,[1]^9,[3]^6\}$ do not exist.
\end{theorem}
In the manuscript \cite{GRWSN} the question was raised whether there exist additional examples of Seidel matrices beyond $\mathfrak{S}$ of order $n=10$ with spectrum $\Lambda(\mathfrak{S})=\{[-3]^4,[2+\sqrt{5}]^3,[2-\sqrt{5}]^3\}$, and its complement $-\mathfrak{S}$, having exactly three distinct eigenvalues, which do not contain a regular graph in their switching class. Our second result answers this question.
\begin{theorem}
Let $S$ be a Seidel matrix of order $n\leq 23$ with exactly three distinct eigenvalues. Then $S$ is equivalent to some Seidel matrix $S'$ so that its ambient graph $\Gamma(S')$ is regular, except for the cases $S\sim\pm\mathfrak{S}$ or $S\sim\pm\left(J_3\otimes(S_6-I_6)+I_{18}\right)$, where $S_6$ is a Seidel matrix of order $6$ with spectrum $\{[-\sqrt{5}]^3,[\sqrt{5}]^3\}$.
\end{theorem}
We have verified by computers that the Seidel matrices $\pm\left(J_{2k+1}\otimes(S_6-I_6)+I_{6(2k+1)}\right)$ do not contain any regular graph in their switching class for $k\in\{1,2\}$. It remains to be seen if this holds for all $k\geq 3$ too. In light of the results tabulated in Table~\ref{table10x}, it is not at all clear whether there exists any further examples (beyond $\pm\mathfrak{S}$) of Seidel matrices with exactly three distinct eigenvalues, having a quadratic eigenvalue, which do not immediately come from the conference two-graphs.
\section{Equiangular lines in $\mathbb{R}^{7}$, $\mathbb{R}^{12}$, and $\mathbb{R}^{14}$}\label{sect5}

The considerations of this paper were motivated in part by the following open problem: does there exist a configuration of $29$ equiangular lines in $\mathbb{R}^{14}$? In the terminology of Seidel matrices, such a system would correspond to a Seidel matrix of order $n=29$ with smallest eigenvalue $\lambda_{\min}=-5$ of multiplicity $15$ (as other potential common angles besides $\alpha=1/5$ can easily be eliminated by results described in \cite{LS}). By Theorem~\ref{tint}, this implies that any sub-Seidel matrix necessarily have smallest eigenvalue at least $-5$, and in particular, sub-Seidel matrices of size $n=15$ have smallest eigenvalue exactly $\lambda_{\min}=-5$. Our initial intuition was that the condition ``$S$ has an eigenvalue $\lambda=-5$'' is extremely strong, and only a few million such matrices of order $15$ are to be found. However, these expectations turned out to be wrong, as inspection of Table~\ref{table3v2} reveals that the number of such matrices grows more rapidly than initially anticipated. While we think that generating the $14\times 14$ Seidel matrices with $\lambda=-5$, or even those with $\lambda_{\min}=-5$ is very much possible with certain additional efforts, classification of the $15\times 15$ Seidel matrices with the same properties, and moreover, augmenting all of those to reach size $n=29$ seems currently out of reach.

Nevertheless, to get some insight into this problem, and in particular, to see how rapidly the number of generated matrices grows in the beginning, and then how fast it declines once we encounter matrices which are large enough to have a guaranteed eigenvalue, we studied two simpler, yet analogous problems.

\subsection{28 lines in $\mathbb{R}^7$ at angle $1/3$}\label{ss53} The first problem we studied was the generation of equiangular lines in $\mathbb{R}^{7}$ with common angle $\alpha=1/3$. Such equiangular line systems correspond to Seidel matrices of order $n\geq 8$ with smallest eigenvalue $\lambda_{\min}=-3$ of multiplicity at least $n-7$. It turned out, that this problem is very easy as the number of Seidel matrices with $\lambda_{\min}\geq-3$ is very limited, see Table~\ref{table3v2}. We were able to generate the counts shown in Table~\ref{table7} effortlessly.

\tiny
\begin{table}[ht]
\[\arraycolsep=4pt
\begin{array}{c|cccccccccccccccccccccc}
\hline
n  & 8  &  9 & 10 & 11 & 12 &  13 &  14 &  15 & 16 & 17 & 18 & 19 & 20 & 21 & 22 & 23 &  24 & 25  & 26  & 27 & 28 & 29\\
\# & 23 & 37 & 54 & 70 & 90 & 101 & 103 & 101 & 90 & 70 & 54 & 37 & 23 & 16 & 10 &  5 &   3 &  2  &  1  &  1 &  1 &  0\\ 
\hline
\end{array}\]
\caption{Seidel matrices with $\lambda_{\min}=-3$ of multiplicity at least $n-7$}\label{table7}
\end{table}
\normalsize

Table~\ref{table7} displays counts regarding the number of inequivalent configurations of $n\geq 8$ equiangular lines in $\mathbb{R}^{7}$ with common angle $1/3$ (for $1\leq n\leq 7$ refer to Table~\ref{table3v2}). It shows that the number of distinct configurations grows at first, reaches its peek at halfway at $n=14$ lines, and then starts to decrease. Uniqueness is reached at $n=26$ lines. It is known that a configuration of $n=29$ lines is impossible \cite{LS}; moreover, $28$ is the maximum number of equiangular lines in $\mathbb{R}^d$ for $7\leq d\leq 13$. This configuration of $n=28$ lines corresponds to a Seidel matrix with spectrum $\{[-3]^{21},[9]^{7}\}$.

\subsection{20 lines in $\mathbb{R}^{12}$ at angle $1/5$}
While the maximum number of equiangular lines is known to be $28$ in $\mathbb{R}^{12}$ (see Section~\ref{ss53}), it was not known prior to this work whether there exists a configuration of $21$ equiangular lines with common angle $\alpha = 1/5$ (see \cite{GRWSN2}). Such a hypothetical configuration would correspond to a Seidel matrix of order $n=21$ with smallest eigenvalue $\lambda_{\min}=-5$ of multiplicity exactly $9$. This in turn implies that any of its $13\times 13$ sub-Seidel matrix would necessarily have smallest eigenvalue $\lambda_{\min}=-5$. The results tabulated in Table~\ref{table3v2} shows that there are exactly $11295930$ such matrices, which should be augmented as described in Section~\ref{sect2}, and then pruned according to the multiplicity of the eigenvalue $-5$ as described in Section~\ref{sect4} until $n=21$ lines are reached. However, no such configuration turned up during the search, and therefore we have the following result.

\begin{theorem}
The maximum number of equiangular lines in $\mathbb{R}^{12}$ with common angle $\alpha=1/5$ is $n=20$.
\end{theorem}

In Table~\ref{table8} we display the counts regarding the number of Seidel matrices of order $n\geq 13$ having an eigenvalue $-5$ of multiplicity at least $n-12$. Note that, for performance reasons, these numbers were not filtered further according to whether or not $-5$ is the smallest eigenvalue. It turns out that each of the $32$ Seidel matrices of order $n=20$ found has in fact smallest eigenvalue $\lambda_{\min}=-5$. This implies, that there exist exactly $32$ distinct configurations of $20$ equiangular lines in $\mathbb{R}^{12}$ with common angle $1/5$. We note that $8$ of these have spectrum $\{[-5]^8,[1]^5,[5]^7\}$, see Table~\ref{table9x}. We remark here that large enough Seidel matrices with smallest eigenvalue exactly $-5$ are equivalent to one whose ambient graph is a Dynkin graph \cite{N}.

\tiny
\begin{table}[h]
\[\begin{array}{c|ccccccccc}
\hline
n   &   13     &   14    &    15   &   16   &   17   &  18   & 19  & 20 & 21\\
 \# & 26030960 & 8897086 & 2931650 & 851892 & 155223 & 16385 & 852 & 32 & 0\\
\hline
\end{array}\]
\caption{Seidel matrices with $\lambda=-5\in\Lambda(S)$ of multiplicity at least $n-12$}\label{table8}
\end{table}
\normalsize

\subsection{28 lines in $\mathbb{R}^{14}$ at angle $1/5$}
The two case studies described in the previous subsections demonstrated that it is possible to classify all equiangular configurations in $\mathbb{R}^d$ for small $d$. The next interesting open case is to decide the maximum number of equiangular lines in $\mathbb{R}^{14}$, which is known to be either $28$ or $29$, see \cite{GRWSN2}. In Section~\ref{ss53} a well-known configuration of $28$ lines in $\mathbb{R}^7$ (hence, in $\mathbb{R}^{14}$) with common angle $1/3$ was obtained as a result of a computer search. Here we recall and analyze a construction with common angle $1/5$. The Seidel matrices of order $36$ with spectrum $\{[-5]^{21},[7]^{15}\}$ were classified in \cite{MSP}. By \cite[Theorem~5.16]{GRWSN2}, removing any sub-Seidel matrix equivalent to $J_8-I_8$ yields a Seidel matrix of order $28$ with spectrum $\{[-5]^{14}, [3]^7, [7]^{7}\}$, which corresponds to a configuration of $28$ equiangular lines in $\mathbb{R}^{14}$. Such a configuration was shown first in \cite{JCT}. It turns out that the $227$ Seidel matrices of order $36$ result in $4009$ Seidel matrices of order $28$, forming $1045$ distinct equivalence classes. We record this enumeration as follows.
\begin{theorem}\label{ttt1}
There exist, up to equivalence, at least $1045$ Seidel matrices of order $28$ with spectrum $\{[-5]^{14}, [3]^7, [7]^{7}\}$. All of these have a regular graph in their switching class.
\end{theorem}
This result shows that during a potential classification of all Seidel matrices of order $n\geq 15$, having an eigenvalue $-5$ of multiplicity at least $n-14$, we will encounter over a $1000$ equivalence classes once we reach order $n=28$.

Finally, we remark that one may ask what are the Seidel matrices of order $n$ having smallest eigenvalue $\lambda_{\min}\geq-5$, and at the same time having largest eigenvalue $\lambda_{\max}\leq c$ for some fixed positive constant $c$. As long as $0<c<5$, such a (computer-aided) classification can be done with reasonable efforts. However, once $c\geq 5$, we will encounter the four inequivalent conference two-graphs of order $n=26$ with spectrum $\{[-5]^{13}, [5]^{13}\}$ (see \cite{MSP}), and once $c\geq 7$ we will encounter the Seidel matrices mentioned in Theorem~\ref{ttt1}. In contrast, inspection of the trace shows that any potential Seidel matrix corresponding to an equiangular line system in $\mathbb{R}^{14}$ with $29$ lines necessarily have an eigenvalue $\lambda\geq \sqrt{437/14}>5.58$, and thus improving over this simple bound computationally with the methods described here seems to be quite difficult.
\section{Verification of the results}\label{sect6}
In order to get confidence in the correctness of the implemented algorithms, we made certain consistency-checks during our computations. An immediate consequence of the orbit-stabilizer theorem \cite[Theorem~3.20]{KO} is the following:
\begin{lemma}\label{l72}
Let $\mathcal{S}$ be a set containing exactly one Seidel matrix of order $n$ up to equivalence. Then
\[\sum_{S\in\mathcal{S}}\frac{1}{|\mathrm{Aut}(S)|}=\frac{2^{n(n-3)/2}}{n!}.\]
\end{lemma}
\begin{proof}
Let us denote by $G$ the group acting on the Seidel matrices of order $n$. The orbit-stabilizer theorem \cite[Theorem~3.20]{KO} yields $\sum_{S\in\mathcal{S}}(|G|/|\mathrm{Aut}(S)|)=2^{n(n-1)/2}$. On the left hand side $|G|/|\mathrm{Aut}(S)|$ counts the total number of distinct Seidel matrices equivalent to $S$. On the right hand side the total number of symmetric, $\pm1$ matrices with $0$ diagonal is shown. Dividing both sides by $|G|=n!2^n$ gives the result.
\end{proof}
Lemma~\ref{l72} gives information on the distribution of the automorphism group sizes of Seidel matrices. We used it to verify that the generated $12886193064$ Seidel matrices of order $n=13$ have automorphism group sizes consistent with the formula described in it. The distribution of the automorphism group sizes, which were obtained as a by-product of the \emph{nauty} calls, is exhibited in Table~\ref{tableauts}. We remark that a variant of the orbit-stabilizer theorem can be applied to verify the counts shown in Tables~\ref{table7} and \ref{table8}, see \cite[Section~10.3]{KO}.

\tiny
\begin{table}
\[\begin{tabular}{rr|rr|rr|rr|rr|rr}
\hline
$|\mathrm{Aut}|$ & \# & $|\mathrm{Aut}|$ & \# & $|\mathrm{Aut}|$ & \# & $|\mathrm{Aut}|$ & \# & $|\mathrm{Aut}|$ & \# & $|\mathrm{Aut}|$ & \#\\
\hline
 12454041600 & 2 & 207360 & 8 & 20736 & 14 & 2880 & 1414 & 512 & 260 & 88 & 12 \\
 159667200 & 2 & 172800 & 4 & 20160 & 78 & 2592 & 24 & 480 & 8412 & 80 & 2280 \\
 43545600 & 2 & 165888 & 2 & 18432 & 12 & 2560 & 2 & 448 & 12 & 72 & 68656 \\
 17418240 & 2 & 161280 & 26 & 17280 & 152 & 2400 & 10 & 432 & 868 & 64 & 196762 \\
 14515200 & 2 & 138240 & 2 & 15552 & 2 & 2304 & 842 & 400 & 8 & 56 & 92 \\
 9676800 & 2 & 120960 & 24 & 15360 & 6 & 2048 & 6 & 384 & 17116 & 52 & 9 \\
 7257600 & 2 & 115200 & 4 & 14400 & 8 & 2016 & 4 & 336 & 36 & 48 & 1638392 \\
 5806080 & 4 & 103680 & 4 & 13824 & 54 & 1920 & 1092 & 320 & 164 & 40 & 4573 \\
 4354560 & 2 & 100800 & 2 & 11520 & 266 & 1728 & 590 & 288 & 23577 & 36 & 120 \\
 2903040 & 2 & 92160 & 6 & 10368 & 12 & 1536 & 394 & 280 & 4 & 32 & 1903876 \\
 2073600 & 2 & 80640 & 22 & 10080 & 16 & 1440 & 768 & 256 & 2458 & 28 & 16 \\
 1935360 & 2 & 69120 & 36 & 9216 & 48 & 1344 & 16 & 240 & 7570 & 24 & 5938352 \\
 1451520 & 4 & 62208 & 6 & 8640 & 96 & 1280 & 6 & 224 & 28 & 20 & 3728 \\
 1209600 & 2 & 60480 & 14 & 7680 & 42 & 1152 & 3194 & 216 & 8 & 16 & 19583940 \\
 967680 & 8 & 57600 & 14 & 6912 & 86 & 1024 & 30 & 192 & 90280 & 12 & 15759127 \\
 806400 & 2 & 55296 & 2 & 6144 & 12 & 1008 & 4 & 168 & 16 & 8 & 192424201 \\
 725760 & 4 & 51840 & 8 & 5760 & 690 & 960 & 4136 & 160 & 623 & 6 & 2744 \\
 645120 & 4 & 46080 & 16 & 5184 & 22 & 864 & 912 & 156 & 1 & 4 & 1704830739 \\
 483840 & 14 & 40320 & 60 & 4800 & 16 & 800 & 4 & 144 & 53588 & 2 & 10943184484 \\
 362880 & 2 & 34560 & 68 & 4608 & 290 & 768 & 2808 & 128 & 21894 &  &  \\
 345600 & 12 & 30720 & 4 & 3840 & 290 & 720 & 22 & 120 & 284 &  &  \\
 322560 & 8 & 28800 & 8 & 3456 & 364 & 672 & 4 & 112 & 72 &  &  \\
 276480 & 4 & 27648 & 14 & 3360 & 4 & 640 & 24 & 104 & 2 &  &  \\
 241920 & 2 & 23040 & 64 & 3072 & 42 & 576 & 10356 & 96 & 393886 &  &  \\
\hline
\end{tabular}\]
\caption{Automorphism group sizes of Seidel matrices of order $13$}\label{tableauts}
\end{table}
\normalsize

\section*{Acknowledgements}
We thank Gary Greaves for some helpful comments.

\appendix
\section{Generation of Seidel matrices via canonical augmentation}
\begin{algorithm}[H]
  \caption{Canonical augmentation of Seidel matrices with custom vertex invariants}
  \begin{algorithmic}[1]
    \Require{A set $\mathcal{S}$ containing the Seidel matrices $S$ of order $n\geq 3$ up to equivalence.}
		\Require{A vertex invariant $f(.)$ with total ordering $\prec$ on its range.}
		\Require{An algorithm to compute a canonical labeling $\sigma$ for a graph, and to determine the $\sim$ equivalence classes of its vertices under the action of the automorphism group.}
    \Statex
    \Function{CanAugInv}{$\mathcal{S}$}
		\Let{$\mathcal{U}$}{$\emptyset$}
		\For{all starting-point matrices $S\in \mathcal{S}$}
		\Let{$\mathcal{C}$}{$\emptyset$}
		\Let{$\mathcal{T}$}{$\emptyset$}
		\For{all possible row (and column) augmentation of $S$ to $\widehat{S}$}
		\State Construct $X(\widehat{S})$ on vertices $\{u_i, v_j^{(k)}\colon i,j\in\{0,\dots, n\}, k\in\{0,1\}\}$
		\Let{$U$}{$\{u_i\colon 0\leq i\leq n; \forall j\in\{0,\dots,n\},j\neq i \implies f(u_j)\neq f(u_i)\}$}
		\If{$U\neq\emptyset$}
		\Let{$p$}{$u_i\in U \colon (\forall j\in U, j\neq i \implies f(u_i)\prec f(u_j)$)}
		\If{$u_n=p$} \Comment{$u_n$ is the most recently appended vertex}
		\State Compute a canonical labeling $Y:=\sigma(X(\widehat{S}))$.
		\If{$Y\notin \mathcal{C}$}
		\Let{$\mathcal{C}$}{$\mathcal{C}\cup\{Y\}$}
		\Let{$\mathcal{T}$}{$\mathcal{T}\cup\{\widehat{S}\}$}
		\EndIf
		\EndIf
		\Else
		\State Compute a canonical labeling $Y:=\sigma(X(\widehat{S}))$.
		\State Determine the vertex orbits under the action of $\mathrm{Aut}(X(\widehat{S}))$.
		\Let{$p$}{$\sigma(0)$} \Comment{$p$ is the vertex with canonical label $0$}
		\If{$u_n\sim p$}
\If{$Y\notin \mathcal{C}$}
		\Let{$\mathcal{C}$}{$\mathcal{C}\cup\{Y\}$}
		\Let{$\mathcal{T}$}{$\mathcal{T}\cup\{\widehat{S}\}$}
		\EndIf
		\EndIf
				\EndIf
			\EndFor
			\Let{$\mathcal{U}$}{$\mathcal{U}\cup\mathcal{T}$}
			\EndFor
			\State\Return $\mathcal{U}$
    \EndFunction
  \end{algorithmic}
\end{algorithm}

\begin{thebibliography}{10}
\bibitem{AM}
\textsc{J. Azarija, T. Marc}:
There is no $(95,40,12,20)$ strongly regular graph,
{\it preprint}, arXiv: 1603.02032v2 [math.CO], 14 pages, (2016).

\bibitem{BMV}
\textsc{E. Bannai, A. Munemasa, B. Venkov}:
The nonexistence of certain tight spherical designs,
{\it Algebra i Analiz},
{\bf 16}, 1--23 (2004).

\bibitem{BUKH2}
\textsc{I. Balla, F. Dr\"axler, P. Keevash, B. Sudakov}:
Equiangular lines and spherical codes in Euclidean space,
{\it preprint},
arXiv:1606.06620v1 [math.CO]

\bibitem{ABARG}
\textsc{A. Barg, W.-H. Yu}:
New bounds for equiangular lines,
{\it Discrete Geometry and Algebraic Combinatorics},
AMS Contemporary Mathematics Series,
{\bf 625}, 111--121 (2014).

\bibitem{BR}
\textsc{G. Brinkmann, J. Goedgebeur, B.D. McKay},
Generation of cubic graphs,
{\it Discrete Mathematics and Theoretical Computer Science},
{\bf 13}, 69--80 (2011).

%
\bibitem{graphspectra2}
\textsc{A.E. Brouwer, E. Spence}:
Cospectral graphs on 12 vertices,
{\it Electron. J. Comb.},
{\bf 16}, \#N20 (2009).

\bibitem{BU}
\textsc{B. Bukh}:
Bounds on equiangular lines and on related spherical codes,
{\it SIAM J. Discrete Math.},
{\bf 30}, 549--554, (2016).

\bibitem{BMS}
\textsc{F.C. Bussemaker, R.A. Mathon, J.J. Seidel}:
Tables of two-graphs, Combinatorics and Graph Theory, 
Lecture Notes in Mathematics,
{\bf 885}, 70--112 (1981).

\bibitem{C}
\textsc{P.J. Cameron}:
Cohomological aspects of two-graphs,
{\it Math. Z.},
{\bf 157}, 101--119 (1977).

\bibitem{DC}
\textsc{D. de Caen}:
Large equiangular sets of lines in Euclidean space,
{\it Electron. J. Comb.},
{\bf 7}, \#55, 3 pp.~ (electronic) (2000).

\bibitem{CD}
\textsc{J. Degraer, K. Coolsaet}:
Classification of some strongly regular subgraphs of the McLaughlin graph,
{\it Discr. Math.},
{\bf 308}, 395--400 (2008).

\bibitem{VD}
\textsc{E.R. van Dam}:
Nonregular graphs with three eigenvalues,
{\it J. Comb. Theory B},
{\bf 73}, 101--118 (1998).

\bibitem{vDS}
\textsc{E.R. van Dam, E. Spence}:
Small regular graphs with four eigenvalues,
{\it Discr. Math.},
{\bf 189}, 233--257 (1998).

\bibitem{FS}
\textsc{M. Fickus, D.G. Mixon and J.C. Tremain}:
Steiner equiangular tight frames,
{\it Linear Algebra Appl.},
{\bf 436}, 1014--1027 (2012).

\bibitem{GEDDES}
\textsc{K.O. Geddes, S.R. Czapor, G. Labahn}:
Algorithms for computer algebra,
Kluwer Academic Publishers, 1992.

\bibitem{RND}
\textsc{E. Ghorbani}:
On eigenvalues of Seidel matrices and Haemers' conjecture,
{\it to appear in Des. Codes Cryptogr.}
doi:10.1007/s10623-016-0248-x, 1--7 (2016).

\bibitem{GODSILXXX}
\textsc{C.D. Godsil, W.L. Kocay}:
Constructing graphs with pairs of pseudo-similar vertices,
{\it J. Combin. Theory B},
{\bf 32}, 146--155 (1982).

\bibitem{GR}
\textsc{C.D. Godsil, A. Roy}:
Equiangular lines, mutually unbiased bases, and
spin models,
{\it European J. Combin.}
{\bf 30}, 246--262 (2009).

\bibitem{GMS}
\textsc{S. Gosselin}:
Regular Two-Graphs and Equiangular Lines, MSc thesis,
Waterloo, Ontario, Canada (2004).

\bibitem{GRWSN}
\textsc{G. Greaves}:
Equiangular line systems and switching classes containing regular graphs,
{\it preprint}, arXiv:1612.03644v2 [math.CO] (2016).

\bibitem{GRWSN2}
\textsc{G. Greaves, J. Koolen, A. Munemasa, F. Sz\"oll\H{o}si}:
Equiangular lines in Euclidean spaces,
{\it J. Combin. Theory A}, 
{\bf 138}, 208--235 (2016).

\bibitem{HAA}
\textsc{J. Haantjes}:
Equilateral point-sets in elliptic two- and three-dimensional spaces,
{\it Nieuw Arch. Wisk.},
{\bf 22}, 355--362 (1948).

\bibitem{INT}
\textsc{W.H. Haemers}:
Interlacing Eigenvalues and Graphs,
{\it Linear Algebra Appl.},
{\bf 227-228}, 593--616 (1995).

\bibitem{graphspectra1}
\textsc{W.H. Haemers, E. Spence}:
Enumeration of cospectral graphs,
{\it European J. Combin.},
{\bf 25}, 199--211 (2004).

\bibitem{HAM}
\textsc{W.H. Haemers}:
Seidel switching and graph energy,
{\it MATCH Commun. Math. Comput. Chem.},
{\bf 68}, 653--659 (2012).

\bibitem{HJ95}
\textsc{R.A. Horn, C.R. Johnson}:
Matrix Analysis (second edition),
Cambridge University Press, Cambridge, (2013).

\bibitem{FL}
\textsc{S.-H. Hou}:
Classroom Note: A Simple Proof of the Leverrier--Faddeev Characteristic Polynomial Algorithm,
{\it SIAM Review},
{\bf 40}, 706--709 (1998).

\bibitem{KO}
\textsc{P. Kaski, P. R.J. \"Osterg\aa rd}:
Classification algorithms for codes and designs,
Springer Berlin, (2006).

\bibitem{LS}
\textsc{P.W.H. Lemmens, J.J. Seidel}:
Equiangular lines,
{\it J. Algebra},
{\bf 24}, 494--512 (1973).

\bibitem{vLS}
\textsc{J.H. van Lint, J.J. Seidel}:
Equilateral point sets in elliptic geometry,
{\it Indag. Math.},
{\bf 28}, 335--348 (1966).

\bibitem{L}
\textsc{V.A. Liskovec}:
Enumeration of Euler graphs,
{\it Vesc\=\i\ Akad. Navuk BSSR Ser. Fiz.-Mat. Navuk.},
{\bf 6}, 38--46 (1970).

\bibitem{MAK}
\textsc{A.A. Makhnev}:
On the nonexistence of strongly regular graphs with parameters $(486,165,36,66)$,
{\it Ukrainian Mathematical Journal},
{\bf 54}, 1137--1146 (2002).

\bibitem{MS}
\textsc{C.L. Mallows, N.J.A. Sloane}:
Two-graphs, switching classes, and Euler graphs are equal in number,
{\it SIAM J. Appl. Math.},
{\bf 28}, 876--880 (1975).

\bibitem{MKHAD}
\textsc{B.D. McKay}:
Hadamard equivalence via graph isomorphism,
Discr. Math.,
{\bf 27}, 213--214 (1979).

\bibitem{MKISO}
\textsc{B.D. McKay}:
Isomorph-free exhaustive generation,
{\it J. Algorithms},
{\bf 26}, 306--324 (1998). [Errata: \url{http://users.cecs.anu.edu.au/~bdm/papers/orderly.pdf}]

\bibitem{MK}
\textsc{B.D. McKay, A. Piperno}:
Practical graph isomorphism II,
{\it J. Symbolic Computation},
{\bf 60}, 94--112 (2013).

\bibitem{MSP}
\textsc{B.D. McKay, E. Spence}:
Classification of regular two-graphs on $36$ and $38$ vertices,
{\it Australas. J. Comb.},
{\bf 24}, 293--300 (2001).

\bibitem{MzK}
\textsc{M. Muzychuk, M. Klin}:
On graphs with three eigenvalues,
{\it Discr. Math},
{\bf 189}, 191--207 (1998).

\bibitem{N}
\textsc{A. Neumaier}:
Graph representations, two-distance sets, and equiangular lines,
{\it Linear Algebra Appl.},
{\bf 114-115}, 141--156 (1989).

%
\bibitem{OKUDA}
\textsc{T. Okuda, W.-H. Yu}:
A new relative bound for equiangular lines and nonexistence of tight spherical designs of harmonic index $4$,
{\it European J. Combin.},
{\bf 53}, 96--103 (2016).

\bibitem{R}
\textsc{R. W. Robinson}:
Enumeration of Euler graphs,
{\it Proof Techniques in Graph Theory}. Academic Press, NY, 147--153 (1969).

\bibitem{SEL}
\textsc{J.J. Seidel}:
Graphs and two-graphs,
{\it Proc. 5th Southeastern Conf.\ on Combinatorics,
Graph Theory, and Computing},
Winnipeg, Canada, Utilitas Mathematica Publishing Inc.\ (1974).

\bibitem{NEWSCREF}
\textsc{T. Soza\'nski}: Enumeration of weak isomorphism classes of signed graphs,
{\it J. Graph Theory},
{\bf 4}, 127--144 (1980).

\bibitem{T}
\textsc{D.E. Taylor}:
Regular $2$-graphs,
{\it Proc. London Math. Soc.}
{\bf 35}, 257--274 (1977).

\bibitem{TTH}
\textsc{D.E. Taylor}:
Some topics in the theory of finite groups,
PhD thesis, University of Oxford (1972).

\bibitem{JCT}
\textsc{J.C. Tremain}:
Concrete constructions of equiangular line sets,
arXiv:0811.2779 [math.MG], (2008).
\end{thebibliography}
\end{document}